\documentclass[reqno]{amsart}
\usepackage{amsmath, amssymb, amsfonts}

\newtheorem{theorem}{Theorem}

\newtheorem{lemma}[theorem]{Lemma}
\newtheorem{proposition}[theorem]{Proposition}
\theoremstyle{remark}
\newtheorem{remark}{Remark}
\theoremstyle{definition}

\newtheorem{example}{Example}
\newtheorem{problem}{Problem}

\newcommand{\CC}{\mathbb C}
\newcommand{\RR}{\mathbb R}

\begin{document}

\title[Tangential holomorphic vector fields]{On the tangential holomorphic
vector fields vanishing at an infinite type point}

\author{Kang-Tae Kim \and Ninh Van Thu}

\thanks{The research of the authors was supported in part by an NRF grant
2011-0030044 (SRC-GAIA) of the Ministry of Education, The Republic of Korea.}

\address{Department of Mathematics and Center for Geometry and its Applications,
Pohang University of Science and Technology,  Pohang 790-784, Republic of Korea}
\email{kimkt@postech.ac.kr, thunv@vnu.edu.vn, thunv@postech.ac.kr}

\address{(Permanent Address of Ninh Van Thu) Department of Mathematics, Vietnam
National University at Hanoi, 334 Nguyen Trai str., Hanoi, Vietnam}
\email{}

\subjclass[2000]{Primary 32M05; Secondary 32H02, 32H50, 32T25.}
\keywords{Holomorphic vector field, real hypersurface, infinite type point.}
\maketitle

\begin{abstract}
Let $(M,p)$ be a $\mathcal C^\infty$ smooth non-Leviflat CR hypersurface germ  in
$\mathbb C^2$ where $p$ is of infinite type. The purpose of this article is to investigate the
holomorphic vector fields tangent to $(M,p)$ vanishing at $p$.
\end{abstract}

\section{Introduction}
A {\it holomorphic vector field} in $\mathbb C^n$ takes the form
$$
X = \sum_{k=1}^n h_k (z) \frac{\partial}{\partial z_k}
$$
for some functions $h_1, \ldots, h_n$ holomorphic in $z=(z_1, \ldots, z_n)$.  A smooth real
hypersurface germ $M$ (of real codimension 1) at $p$ in $\mathbb C^n$ takes a defining
function,
say $\rho$, such that $M$ is represented by the equation $\rho(z)=0$.  The holomorphic vector
field $X$ is said to be {\it tangent} to $M$ if its real part $\hbox{Re }X$ is tangent to $M$, i.e.,
$X$ satisfies the equation $\hbox{Re }X\rho = 0$.

In several complex variables, such tangential holomorphic vector fields arise naturally from the
action by the automorphism group of a domain. If $\Omega$ is a smoothly bounded domain in
$\CC^n$ and if its automorphism group $\hbox{Aut }(\Omega)$ contains a 1-parameter subgroup,
say $\{\varphi_t\}$, then the $t$-derivative generates a holomorphic vector field.  In case the
automorphisms of $\Omega$ extend across the boundary (cf., \cite{Fef}, \cite{Bell-Lig}),
the vector field generated as such
becomes a holomorphic vector field tangent to the boundary hypersurface $\partial\Omega$.
Even such a rough exposition illustrates already that the study of such vector fields is closely
linked with the study of the automorphism group of $\Omega$, an important research
subject in complex geometry.

Over the decades, the domains admitting such automorphism groups with a boundary accumulating 
orbit have been studied extensively by many authors. To take only a few examples, 
well-known theorems such as the Wong-Rosay theorem \cite{W, R}, the Bedford-Pinchuk theorems
\cite{B-P1, B-P2, B-P3} and the theorems characterizing the bidisc by
Kim, Pagano, Krantz and Spiro \cite{Kim, K-P, K-K-S} gave characterization of the bounded
domain with non-compact automorphism group among many theorems in this circle of research.  All these
theorems rely upon the existence of an orbit of an interior point by the action of the
automorphism group accumulating at a pseudoconvex boundary point, strongly pseudoconvex,
of D'Angelo finite type \cite{D}, or of Levi flat in a neighborhood, respectively. For the complementary cases, Greene and Krantz posed a conjecture that for a smoothly
bounded pseudoconvex domain admitting a non-compact automorphism group, the point orbits can
accumulate only at a point of finite type \cite{GK}.

In the case that the automorphism group extends to a subgroup of the diffeomorphism group of the closure and that the automorphism group of a bounded domain has a nontrivial connected subgroup whose point orbit accumulates at a boundary point, it produces an action on the boundary surface by  a nontrivial tangential holomorphic vector field vanishing at the boundary accumulation point.
Analysis of such vector fields has turned out to be quite essential: cf., e.g., \cite{B-P1, B-P2, B-P3} in which the existence of parabolic vector fields plays an important role.  In case the vector field is contracting at a $\mathcal C^\infty$ smooth boundary point, a theorem of Kim and Yoccoz \cite{KiYo} implies that the boundary point is of finite type, thus solving an important case of the Greene-Krantz conjecture mentioned above.  Therefore, the following problem emerges naturally:

\begin{problem}
 Assume that $(M,p)$ is a non-Leviflat CR hypersurface germ in $\mathbb C^n$ such that $p$ is a point of infinite type.  Characterize all holomorphic vector fields tangent to $M$ vanishing at $p$.
\end{problem}

A typical consequence of the main results of this paper is as follows:

\begin{theorem}  \label{T0}
Let $(M,0)$ be a pseudoconvex $\mathcal C^\infty$ CR hypersurface germ in $\mathbb C^2$
defined by $\hbox{Re }z_1 + P(z_2) = 0$, where $P(z_2)$ satisfies:
\begin{itemize}
\item[(1)] $P(z_2)$ vanishes to infinite order at $z_2=0$,
\item[(2)] $P(z_2) >0$ for any $z_2 \not= 0$.
\end{itemize}
If $X$ is a holomorphic tangent vector field to $(M,0)$ vanishing at $0$, then
$X$ is either identically zero, or $X=i\alpha z_2 \partial/\partial z_2$ with
$\alpha$ a nonzero real constant, in which case $P(z_2) = P(|z_2|)$.
\end{theorem}

The defining function of a general CR hypersurface germ $M$, say, at $0$ even in
complex dimension 2 is more complicated.  Let $(M,0)$ be a CR hypersurface germ
at the origin $0$ in $\CC^2$ where $0$ is of infinite type.  If one writes $z_1 = u+iv$, then $M$ takes a defining function equation
$$
\rho (z) = u + P(z_2) + v Q(z_2, v) = 0
$$
from the Taylor expansion (of $u$) in the variable $v$. Despite its general feature, the theorem above is clearly just a special case of the main result of this article; we indeed present the complete list of tangential holomorphic vector fields vanishing at $0$ for much broader a class of CR hypersurfaces in $\CC^2$.

Before going in further we acknowledge that this work has been heavily influenced by many papers preceding ours.  Some of them, in addition to the ones cited already, include \cite{Ber, By, By1, GK, IK, Ka, Ki1, Ki2, Ki3, La, Ninh}, just to name a few.
We point out that the results of this paper encompass almost all cases in the literature [{\it Op.\ cit.}] and in fact more general.

\section{Main results of this paper}

For the sake of smooth exposition, we would like to explain the main results of
this article, deferring the proof to the later sections.
\smallskip

Let $M$ be a $\mathcal{C}^\infty$-smooth real hypersurface germ at the origin $0=(0,0)$ in
$\mathbb C^2$.  Then it admits the following expression:
\begin{equation} \label{eq}
M= \{(z_1,z_2)\in \mathbb C^2: \rho(z_1,z_2)
=\text{Re }z_1+P(z_2)+
\text{Im }z_1Q(z_2,\text{Im }z_1)=0\},
\end{equation}
where $P$ and $Q$ are $C^\infty$-smooth functions with $P(0)=0, dP(0)=0$ and $Q(0,0)=0$.
We now  discuss what the concept of infinite type means.
\medskip

Following \cite{D}, we consider a smooth real-valued function $f$ defined in a neighborhood of
$0$ in $\mathbb C$.  Let $\nu(f)$ denote the order of vanishing of $f$ at $0$, by the first
nonvanishing degree term in its Taylor expansion at $0$.  Order 1 vanishing simply means $f(0)=0$,
but the first degree term is not identically zero, for instance.

In case $f$ is a mapping into $\RR^k$, $k > 1$, we consider the order of vanishing of all the
components and take the smallest one among them for the vanishing order of $f$.  Denote it by
$\nu_0 (f)$. Also denote by $\Delta = \{z \in \CC \colon |z|<1\}$. Then the origin is called a {\it
point of infinite type} if, for every integer $\ell > 0$, there exists a holomorphic map
$h:\Delta \to \CC^2$ with $h(0)=(0,0)$ such that
$$
\nu_0 (h) \not=\infty \hbox{ and } \frac{\nu_0 (\rho \circ h)}{\nu_0(h)} > \ell.
$$
Notice that the terminology ``infinite type'' coincides with ``not of D'Angelo finite type'', since
the definition of $0$ being a point of $M$ of D'Angelo finite type is  that the supremum of
$\nu_0 (\rho \circ h)/\nu_0(h)$ over all possible analytic curves $h$ is bounded.  If we just
call this supremum the D'Angelo type of $M$ at $0$, denoted by $\tau(M,0)$, then the
definition of infinite type is simply that $\tau(M,0)=\infty$.
\medskip

Then the following result pertaining to the infinite type is our first result of this article:

\begin{theorem}\label{T1}
Suppose that $M$ is a smooth real hypersurface germ in $\CC^2$ at the origin defined by
$$
\rho(z_1,z_2)=\text{Re }z_1+P(z_2)+ \text{Im }z_1Q(z_2,\text{Im }z_1)=0.
$$
Assume also that  $\displaystyle{\frac{\partial^N P}{\partial z_2^N}\Big|_{z_2=0} =0}$
for  every nonnegative integer $N$. Then the origin is a point of infinite type
if and only if  $P(z_2)$ vanishes to infinite order at $z_2=0$.
\end{theorem}

Notice that the condition that $\frac{\partial^N P}{\partial z_2^N}\Big|_{z_2=0}=0$ for
every positive integer $N$ for $P$ is not an artificial restriction.  In the viewpoint
of formal power series expansion of $P$ at the origin, this condition simply amounts to
that each homogeneous polynomial of homogeneous degree does not contain any harmonic terms.
This can be achieved
through a holomorphic change of the coordinate system at the origin.
\medskip

Then we present the following  characterization of holomorphic vector fields which
are tangent to a hypersurface and vanish at an infinite type point.

\begin{theorem}\label{T2}
If a hypersurface germ $(M,0)$ is defined by the equation
$\rho(z) := \rho(z_1,z_2)=\text{Re }z_1+P(z_2)+ (\text{Im }z_1)~ Q(z_2,
\text{Im }z_1)=0$,
satisfying the conditions:
\begin{itemize}
\item[(1)] $P(z_2)>0$ for any $z_2 \not= 0$, %
\item[(2)] $P$ vanishes to infinite order at $z_2=0$,
 and
\item[(3)] $\displaystyle{\frac{\partial^N Q(z_2,0)}{\partial z_2^N}\Big|_{z_2=0}=0}$ for every
positive integer $N$,
\end{itemize}
then any holomorphic vector field vanishing at the origin tangent to $(M,0)$ is either identically
zero, or of the form $i \beta z_2\frac{\partial}{\partial z_2}$ for some non-zero real number
$\beta$, in which case it holds that $\rho(z_1,z_2)=\rho(z_1,|z_2|)$.
\end{theorem}

Note that Theorem \ref{T2} implies Theorem \ref{T0}.

\begin{remark}
It is worth noting that the conclusion of Theorem \ref{T2} says that there are no hyperbolic or parabolic orbits of CR automorphisms of $(M,0)$ accumulating at $0$. This is seen by working out the necessary analytic differential equation associated with the vector field $H$.
\end{remark}

\begin{remark}
As to the hypothesis of the theorem, the condition (1) is not unnatural; this condition holds for instance, up to a change of the holomorphic coordinate system, if $(M,0)$ admits a holomorphic peak function at $0$. Condition (2) simply says that $0$ is a point of infinite type. The last condition (3) is the only technical condition but is essential for the conclusion of the theorem. Of course a holomorphic change of coordinates can remove the harmonic terms from $Q(z_2, 0)$, but then the new remaining term does no longer possess the factor $\text{\rm Im }z_2$. In such a case, we show by the example below that, without the condition (3), the conclusion of the theorem does not hold. On the other hand, the condition (3) is used only once in the proof, i.e., in Section 4.2. There, we need only that $Q(z_2,0)$ does not contain the monomial term $z_2^k$, in case $k=\nu(Q(z_2,0))$, finite.
\end{remark}

\begin{remark}[The notation $P'$]
Taking the risk of confusion we employ the notation
$$
P'(z_2) = P_{z_2} (z_2) = \frac{\partial P}{\partial z_2} (z_2)
$$
throughout the paper.  Of course for a function of single real variable $f(t)$, we shall continue using $f'(t)$ for its derivative, as well.
\end{remark}

\begin{example} We now demonstrate that there exists a hypersurface germ $(M,0)$ satisfying
the hypotheses of Theorem \ref{T2} except the condition (3), which admits a nontrivial holomorphic tangent
vector field with both $\partial/\partial z_1$ and $\partial/\partial z_2$ present in the expression nontrivially.

Let $M$ be the real hypersurface in $\Delta^2\subset\mathbb C^2$ defined by
$$
M=\{(z_1,z_2)\in \Delta^2: \rho(z_1,z_2)=\text{Re}z_1+P(z_2)+(\text{Im }z_1 )Q(z_2)=0\},
$$
where $P$ and $Q$ are given as follows:
$$
Q(z_2)=\tan \big((\text{Im } z_2)^2\big)
$$
and
\begin{equation*}
P(z_2)=
\begin{cases}
\exp \big(-\frac{1}{|z_2|^2}+\frac{1}{2}\text{Im }(z^2_2)-\log |\cos ((\text{Im}z_2)^2)|\big)
& \text{if}\; 0<|z_2|<1\\
0 & \text{if} \; z_2=0.
\end{cases}
\end{equation*}
Define a holomorphic vector field $H$ by
$$
H=z_1z_2^2\frac{\partial}{\partial z_1}+iz_2\frac{\partial}{\partial z_2}.
$$
We claim that the holomorphic vector field $H$ is tangent to the hypersurface $M$. Indeed,
computation shows:
\begin{equation}\label{eq01}
\begin{split}
&\text{Re}\Big[i z_2Q_{z_2}(z_2)+\Big(\frac{i}{2}-\frac{Q(z_2)^2}{2i}\Big)z_2^2\Big]=0,\\
&\text{Re}\Big[iz_2 P_{z_2}(z_2) -\Big(\frac{1}{2}+\frac{Q(z_2)}{2i}\Big)z_2^2
P(z_2)\Big]=0.
\end{split}
\end{equation}
Moreover, $\rho_{z_1}(z_1,z_2)=\frac{1}{2}+\frac{Q(z_2)}{2i}$ and
 $\rho_{z_2}(z_1,z_2)=P'(z_2)+(\text{Im} z_1)Q_{z_2}(z_2)$. Therefore it follows by
(\ref{eq01}) that
\begin{equation}\label{eq02}
\begin{split}
\text{Re}H(\rho(z_1,z_2))&=\text{Re}\Big[\Big(\frac{1}{2}+\frac{Q(z_2)}{2i}\Big) z_1z_2^2
+(P'(z_2)+\big(\text{Im} z_1)Q_{z_2}(z_2)\big) i z_2\Big]\\
&= \text{Re}\Big[\Big(\frac{1}{2}+\frac{Q(z_2)}{2i}\Big) \Big(i (\text{Im}z_1)-
P(z_2)-(\text{Im}z_1) Q(z_2)\Big)z_2^2\\
&\qquad +\big(P'(z_2)+(\text{Im} z_1)Q_{z_2}(z_2)\big) i z_2\Big]\\
&=\text{Re}\Big[iz_2 P'(z_2) -\Big(\frac{1}{2}+\frac{Q(z_2)}{2i}\Big)z_2^2
P(z_2)\Big]\\
&\qquad +(\text{Im} z_1)\text{Re}\Big[iz_2 Q_{z_2}(z_2) +\Big(\frac{i}{2}-\frac{Q(z_2)^2}
{2i}\Big)z_2^2 \Big]\\
&=0
\end{split}
\end{equation}
 for every $(z_1,z_2)\in M$.  Hence the claim is justified.
\end{example}

\section{On the defining equations for the germs of infinite type}

From here on, the vanishing order is always computed at the origin.  Henceforth, the notation $\nu$ will
represent $\nu_0$, unless mentioned otherwise.

\subsection{Proof of Theorem \ref{T1}}
%
%
Assume that $P(z_2)$ vanishes to infinite order at $z_2=0$. Then define $\varphi$ to be
the holomorphic curve $\varphi(t)=(0,t)\colon\Delta \to \CC^2$, where $\Delta = \{z \in \CC \colon |z|<1\}$. Then $\nu(\rho\circ \varphi)=\nu(P)=+\infty$
and consequently,
$$
\tau(M,0)=\sup_{\varphi} \frac{\nu(\rho\circ \varphi)}{\nu(\varphi)}=+\infty.
$$
\smallskip

In order to establish the converse, suppose that $\tau(M,0)=+\infty$.
Then for each $N>1$ there is a holomorphic curve $\varphi_N:\Delta\to\mathbb C^2$
with $\varphi_N(0)=(0,0)$
such that
$$
\frac{\nu(\rho\circ \varphi_N)}{\nu(\varphi_N)}\geq N.
$$
The present goal is to show that $\nu(P)\geq N$.
\smallskip

For convenience, we use temporarily the notation
$$
\varphi_N(t) = (z_1(t), z_2(t))
$$
where $t$ is the complex variable.
Consider
\begin{equation}\label{eqt1}
\rho\circ\varphi_N(t)=\text{Re }z_1(t)+P(z_2(t))+\text{Im }z_1(t) Q(z_2(t),\text{Im }z_1(t)).
\end{equation}

The vanishing order of the third term of the right-hand side of (\ref{eqt1}) is strictly larger than
the first. Thus the third term does not have any role in the type consideration. Thus we consider the
following three cases:
\medskip

\bf Case 1.\ {\mathversion{bold}  $\bf \nu(P(z_2))<\nu(z_1)$}: \rm
If $\nu(z_1)>\nu(\varphi_N)$,
then $\nu(z_2)=\nu(\varphi_N)$. Hence $z_2\not \equiv 0$. Moreover
$$
N\leq \frac{\nu(\rho\circ \varphi_N)}{\nu(\varphi_N)}
=\frac{\nu(P(z_2))}{\nu(z_2)}=\nu(P),
$$
as desired.

The remaining subcase to consider is when $\nu(z_1)=\nu(\varphi_N)$.  In this case,
$\nu(z_2)\geq \nu(z_1)$, since
$\nu(\varphi_N(t)) = \min \{\nu(z_1(t)), \nu(z_2(t))\}$. In particular,
$z_1\not \equiv 0$. And, one obtains that
$$
N\leq \frac{\nu(\rho\circ \varphi_N)}{\nu(\varphi_N)} =
\frac{\nu(P(z_2))}{\nu(z_1)}<\frac{\nu(z_1)}{\nu(z_1)}=1.
$$
But this is absurd.  Hence our goal is justified in this case.
\medskip

\bf Case 2.\ {\mathversion{bold} $\bf \nu(P(z_2))> \nu(z_1)$}: \rm
If $\nu(z_2)\leq \nu(z_1)$,
then $\nu(z_2)=\nu(\varphi_N)$. Hence $z_2\not \equiv 0$, and
$$
N\leq \frac{\nu(\rho\circ \varphi_N)}{\nu(\varphi_N)}=\frac{\nu(z_1)}{\nu(\varphi_N)}
<\frac{\nu(P(z_2))}{\nu(z_2)}=\nu(P),
$$
as desired.

The remaining subcase, now, is when $\nu(z_2)>\nu(z_1)$. In this case
$\nu(\varphi_N)= \nu(z_1)$. Then $z_1\not \equiv 0$, and
$$
N\leq \frac{\nu(\rho\circ \varphi_N)}{\nu(\varphi_N)}=\frac{\nu(z_1)}{\nu(z_1)}=1,
$$
which is absurd. Hence the claim is prove in this case also.
\medskip

\bf Case 3.\ {\mathversion{bold} $\nu(P(z_2))=\nu(z_1)$}: \rm
If $\nu( \text{Re }z_1(t)+P(z_2(t)))=\nu(z_1(t))$,
then we also obtain $\nu(P)\geq N$ by repeating the arguments as above .

Thus the only remaining case is when $\nu( \text{Re }z_1(t)+P(z_2(t)))>\nu(z_1(t))$.
In such instance,  $z_1(t)\not \equiv 0$, $z_2(t)\not \equiv 0 $, and $\nu(P)<+\infty$. It follows
then that $z_1(t)=a_m t^m +o(t^m) $ and that $z_2(t)=b_n t^n +o(t^n) $,
where $m,n\geq 1, a_m\ne 0, b_n\ne 0$. Moreover we may also write $P(z_2)=\psi(z_2)+ ...$,
where $\psi$ is a nonzero real homogeneous polynomial of finite degree, say, $k$ with $k\geq 2$.
Since $\nu( \text{Re }z_1(t)+P(z_2(t)))>\nu(z_1(t))=\nu(P(z_2))$, one sees that
$m=nk$ and
$$
\text{Re }(a_m t^m)+\psi(b_n t^n)=0, $$
for every $t$ in a neighborhood of $0$ in $\mathbb C$. Letting $s=b_n t^n$,
we arrive at $\psi(s)=\text{Re }(\frac{a_m}{b_n^k} s^k)$.  But this is impossible since no finite
order jet of $P$ can contain any nonzero harmonic term.
\smallskip

Altogether, the proof of Theorem \ref{T1} is complete. \hfill $\Box$

\subsection{On the non-Leviflat hypersurface germs at $0$ of infinite type}
Unlike the finite type case, it has not very well been clarified in the case of
infinite type whether there is a variety that has infinite order contact with the hypersurface germ in
consideration.  We present a discussion concerning this point.  We begin with the following
which generalizes Lemma 2.2 of \cite{Kol}.

\begin{proposition}
If $\tau(M,0)=+\infty$, then there is a sequence
 $\{a_n\}_{n=2}^\infty \subset \mathbb C$ such that for each integer $N\geq 2$
such that the holomorphic curve $\varphi_N (t) = (z_1(t), z_2(t))$ defined by
$$
z_1(t)= -\sum_{j=2}^N a_j t^j,  \quad z_2(t)=t
$$
satisfies
$$
\nu(\rho\circ \varphi_N) \geq N.
$$
\end{proposition}

\begin{proof}
We start with the second order terms; $\rho $ as
$$
\rho(z)= \text{Re } z_1+\psi(z_2)+o(|z_2|^2, \text{Im } z_1),
$$
where $\psi$ is a real valued homogeneous polynomial of degree $2$.
Since $\tau(M,0)=+\infty$, the proof-argument of
Theorem \ref{T1} implies that $\psi(z_2)=\text{Re} (a_2 z_2^2)$.
Let $\Phi_2:(z_1,z_2) \mapsto (\xi_1, \xi_2)$ be an automorphism of $\mathbb C^2$
defined by $ \xi_1=z_1+ a _2 z_2^2 , \xi_2=z_2$ (a ``shear'' mapping).
Then
$$
\rho\circ \Phi_2^{-1}(\xi_1, \xi_2)=\text{Re}\xi_1+o(|\xi_2|^2,\text{Im } \xi_1 ).
$$
Now proceed by induction: Assume that, for each $j>2$, the coefficients $a_2,\cdots, a_{j-1}$
and the automorphisms $\Phi_2,\cdots, \Phi_{j-1}$ have already been determined so that
$$
\rho\circ \Phi^{-1}_2\circ\cdots\circ \Phi_{j-1}^{-1}(\zeta)=\text{Re } \zeta_1+\psi_j(\zeta_2)
+o(|\zeta_2|^j, \text{Im } \zeta_1),
$$
where $\psi_j$ is either 0 or a real valued homogeneous polynomial of degree $j$.

The the proof-argument of Theorem \ref{T1} implies that
$\psi_j(z_2)=\text{Re} (a_j z_2^j)$. Thus let $\Phi_j:\mathbb C^2\to \mathbb C^2$
be the automorphism of  $\mathbb C^2$ defined by
$$
\xi_1=\zeta_1+ a_j \zeta_2^j,\; \xi_2=\zeta_2.
$$
We then obtain
$$
\rho\circ \Phi_2^{-1}\circ\cdots\circ \Phi_{j-1}^{-1}\circ \Phi_j^{-1}(\xi)=
\text{Re}\xi_1+o(|\xi_2|^j,\text{Im } \xi_1 ).
$$

This induction argument yields the sequence $\{a_k\}_{k=2}^\infty\subset \mathbb C$.
Furthermore, for each $N\geq 2$, a non-singular holomorphic curve $\varphi_N$ defined
on a neighborhood of $t=0$ in $\mathbb C$ by
$\varphi_N(t):= \Phi_2^{-1}\circ\cdots\circ \Phi_{N}^{-1}(0,t)$ which satisfies
$\rho\circ \varphi_N(t)=o(|t|^N)$, or equivalently
$\nu(\rho\circ \varphi_N)\geq N$.  Of course it is clear that
$$
\varphi_N (t) = \Big(-\sum_{j=2}^N a_j t^j, ~t\Big),
$$
and the proof is complete.
\end{proof}

Note that if the series $\sum_{j=2}^{\infty}a_j z^j$ converges in
an open neighborhood of $z=0$ in the complex plane, then
$\nu(\rho\circ \varphi_\infty)=+\infty$,
where $\varphi_\infty$ is the holomorphic curve given on a neighborhood of $t=0$
in $\mathbb C$ by
$$
z_1(t)= -\sum_{j=1}^\infty a_j t^j,\quad z_2(t)=t .
$$

So it is natural to ask at this point whether there exists a regular holomorphic curve
$\varphi_\infty$ defined on a neighborhood of the origin in the complex plane such that
 $\nu(\rho\circ \varphi_\infty)=+\infty$, or even more bold to ask whether the above
procedure may produce such curve. The following example gives the negative answer.

\begin{example}
\it There exists a hypersurface germ $(M,0)$ with $\tau(M,0)=+\infty$ that does not admit any
regular holomorphic curve that has infinite order contact with $M$ at $0$. \rm
\smallskip

The construction is as follows: for $n=2,3,\cdots$, denote by
$$
g_n(t)=\frac{1}{t^n-a_n}+\frac{1}{a_n},
$$
a function of the single complex variable $t$ with
for $|t|<1/n$, where $a_n=2/n^n$. Then $g_n$ is holomorphic on
$\{|t|<1/n\}$ with $\nu(g_n)=n$. Expanding $g_n$ into Taylor series we obtain
$$
g_n(t)=\frac{1}{a_n}-\frac{1}{a_n} \frac{1}{1-t^n/a_n}=-\sum_{k=1}^\infty
 \frac{1}{a^{k+1}_n}t^{nk}.
$$
Then
$$
g^{(j)}_n(0)=
\begin{cases}
-\dfrac{(nk)!}{a^{k+1}_n}  & \text{if}\; j=nk \text{ for some integers }  k\geq 1, n\geq 2 \\
0    &\text{otherwise}.
\end{cases}
$$
For each $n=2,3,\cdots$ denote by $\tilde f_n(z)$ the $\mathcal C^\infty$-smooth
function on $\mathbb C$ such that
$$
\tilde f_n(z)=
\begin{cases}
\text{Re}( g_n(z))      & \text{if } |z|<1/(n+1)  \\
0                       & \text{if } |z|>1/n.
\end{cases}
$$
Of course, $\nu(\tilde f_n)=n$ and
$$
\frac{\partial^j}{\partial z^j}\tilde f_n(0)=
\begin{cases}
-\dfrac{1}{2}\dfrac{(nk)!}{a^{k+1}_n} & \text{if } j=nk \text{ for some integers }  k\geq 1,
n\geq 2 \\
0 & \text{otherwise}.
\end{cases}
$$
Denote by $\{\lambda_n\}$ an increasing sequence of positive numbers such that
$$
\lambda_n\geq\max\Big\{1, \Big\|\frac{\partial^{k+j} \tilde f_n}{\partial z^k\partial \bar{z}^j}
 \Big\|_\infty \colon j,k\in \mathbb N, k+j\leq n\Big\},
$$
where $\| ~\|_\infty$ represents the supremum norm.  Now
let $f_n(z)=\frac{1}{n^n \lambda_n^n}\tilde f_n(\lambda_n z)$.
The repeated use of the chain rule implies that
$$
\frac{\partial^k f_n}{\partial z^k}(z)=\frac{1}{n^n \lambda_n^{n-k}}
\frac{\partial^k \tilde f_n}{\partial z^k}(\lambda_n z),\; k=0,1,\cdots.
$$
Combining this with the previous result for the $k$-th derivative of $\tilde f_n$ at zero, one
arrives at
$$
\frac{\partial^k f_n}{\partial z^k}(0)=
\begin{cases}
-\dfrac{n! n^n}{8} & \text{if } k\mid n  \\
0                               & \text{if } k\nmid  n.
\end{cases}
$$
Let $f(z)=\sum_{n=2}^{\infty}f_n(z)$. For every $k,j$, non-negative integers, one sees that
\begin{equation*}
\begin{split}
 \sum_{n=2}^{\infty}\Big\|\frac{\partial^{k+j} f_n}{\partial z^k\partial \bar {z}^j}(z)\Big\|_\infty
&\leq  \sum_{n=2}^{j+k}\frac{1}{n^n\lambda_n^{n-k-j}}\Big\|
\frac{\partial^{k+j} \tilde f_n}{\partial z^k\partial \bar {z}^j}(z)\Big\|_\infty\\
&+ \sum_{n=j+k+1}^{\infty}\frac{1}{n^n\lambda_n^{n-k-j-1}}
\dfrac{\|\frac{\partial^{k+j} \tilde f_n}{\partial z^k\partial \bar {z}^j}(z)\|_\infty}{\lambda_n}\\
&\leq  \sum_{n=2}^{j+k}\frac{1}{n^n\lambda_n^{n-k-j}}\Big\|
\frac{\partial^{k+j} \tilde f_n}{\partial z^k\partial \bar {z}^j}(z) \Big\|_\infty
+\sum_{n=j+k+1}^{\infty} \frac{1}{n^n} \\
&<+\infty.
\end{split}
\end{equation*}
This shows that $f\in\mathcal{C}^\infty(\mathbb C)$.

Let $\{p_n\}_{n=1}^\infty$ be a sequence of prime numbers such that $p_n\to+\infty$
as $n\to \infty$. It is easy to see that
\begin{equation*}
\begin{split}
 \frac{\partial^{p_n} }{\partial z^{p_n}}f(0)&=\sum_{j=2}^\infty
\frac{\partial^{p_n} }{\partial z^{p_n}}f_j(0)\\
&=\sum_{j=2}^{p_n-1}\frac{\partial^{p_n} }{\partial z^{p_n}}f_j(0)+
\frac{\partial^{p_n} }{\partial z^{p_n}}f_{p_n}(0)+\sum_{j=p_n+1}^{p_n-1}
\frac{\partial^{p_n} }{\partial z^{p_n}}f_j(0)\\
&=\frac{\partial^{p_n} }{\partial z^{p_n}}f_{p_n}(0)=-\frac{p_n! (p_n)^{p_n}}{8}.
\end{split}
\end{equation*}

The hypersurface germ $M$ at $(0,0)$ we consider is defined by
$$
M=\{(z_1,z_2)\in \mathbb C^2: \rho=\text{Re} z_1+f(z_2) =0\}.
$$
We are going to show that $\tau(M,0)=+\infty$. For this purpose, for each $N\geq 2$,
 consider $\varphi_N=(z_1,z_2)$, a holomorphic curve defined on
$\{t\in \mathbb C:|t|<\dfrac{1}{\lambda_N( N+1)}\}$ by
$$
z_1(t)= -\sum_{n=2}^N \frac{1}{n^n\lambda_n^n}g_n(\lambda_n t); z_2(t)=t .
$$
Then $\rho\circ \varphi_N(t)=\sum_{n=N+1}^\infty f_n(t)$.
Since $\nu(f_n)=n$ for $n=2,3\cdots$, it follows that $\nu(\rho\circ \varphi_N)=N+1$,
and hence $\tau(M,0)=+\infty$.

We finally demonstrate that there is no regular holomorphic curve $\varphi_\infty(t)=(h(t),t)$,
such that $\nu(\rho\circ \varphi_\infty)=+\infty$.

Assume the contrary that such a holomorphic curve exists.
Then $\rho\circ \varphi_\infty(t)= \text{Re } h(t)+f(t)=o(t^N)$ for every $N=2,3,\cdots$,
and thus $h^{(N)}(0)= -2\frac{\partial^{N} }{\partial z^{N}}f(0)$ for any $N=0,1,\cdots$.
Consequently, $h^{(p_n)}(0)=-2\dfrac{\partial^{p_n} }{\partial z^{p_n}}f(0)
=\dfrac{(p_n!) (p_n)^{p_n}}{4}$, and moreover
\begin{equation*}
\begin{split}
\limsup_{N\to\infty}\sqrt[N]{\frac{|h^{(N)}(0)|}{N!}}
&\geq \limsup_{p_n\to\infty}\sqrt[p_n]{\frac{|h^{(p_n)}(0)|}{(p_n)!}}\\
&=\lim_{p_n\to\infty}\sqrt[p_n]{\frac{(p_n !) (p_n)^{p_n}}{4(p_n !)}}=
\lim_{p_n\to \infty} \frac{p_n}{\sqrt[p_n]{4}}=+\infty.
\end{split}
\end{equation*}
This implies that the Taylor series of $h(z)$ at $0$ has radius of convergence 0, which is impossible since $h$ is holomorphic in a neighborhood of the origin.
This ends the proof. \hfill $\Box$
\end{example}

\section{Analysis of holomorphic tangent vector fields}

This section is entirely devoted to the proof of Theorem \ref{T2}.
\medskip

\setcounter{theorem}{0}
Let $M=\{(z_1.z_2) \in \CC^2 \colon \hbox{Re }z_1 + P(z_2) + (\hbox{Im }z_1)\ Q(z_2,
\text{Im }z_1)=0\}$ be the real hypersurface germ at $0$ described in the hypothesis of Theorem
\ref{T2}.
Our present goal is to characterize its holomorphic tangent vector fields.
\smallskip

For the sake of smooth exposition, we shall present the proof in two subsections. In 4.1, several technical lemmas are introduced and proved. Then the proof of Theorem \ref{T2} is presented in 4.2.

\subsection{Technical lemmas} Start with

\begin{lemma} \label{l1}
If $\beta$ is a real number, then
\begin{equation}
\label{eq1}
\lim_{z\to 0} \text{Re} \Big[1 +i \beta z \frac{P'(z)}{P(z)^{n+1}}\Big] \not= 0
\end{equation}
for any nonnegative integer $n$.
\end{lemma}

\begin{proof}
We may assume that $\beta \ne 0$ as there is nothing to prove otherwise.  Suppose that
$$
\lim_{z\to 0} \text{Re} \Big[1 +i \beta z \frac{P'(z)}{P(z)^{n+1}}\Big] = 0.
$$
Then let
\begin{equation*}
F(z):=
\begin{cases}
\dfrac{1}{2}\log P(z)  &\text{if  } n=0 \\
\dfrac{-1}{2n P^n(z)}  & \text{if  } n\geq 1.
\end{cases}
\end{equation*}
Also let $u(t):=F(r e^{i\beta t }),\; t\in (-\infty,+\infty)$, for some $r>0$ is sufficiently small.
Then (\ref{eq1}) implies that
$u'(t)=-1+\gamma(r e^{i\beta t}), \; t\in (-\infty,+\infty)$. Let $r_0>0$ be
such that $|\gamma(re^{i\beta t})|<1/(2|\beta|)$ for all $r<r_0$ and $t\in (-\infty,+\infty)$.
Now for a fixed number $r$ with $0<r<r_0)$ we have
$u(t)-u(0)=-t + \int_0^t \gamma(re^{i\beta s})ds$ for $t\in (-\infty,+\infty)$. Thus
$0=|u(2\pi/\beta)-u(0)|\geq 2\pi/|\beta|-\int_0^{2\pi/\beta}|\gamma(re^{i\beta s})|ds
\geq\pi/|\beta|$,
which is impossible. Hence the lemma is proved.
\end{proof}

\begin{lemma}\label{l22}
Denote the punctured unit disc by $\Delta^*_{\epsilon_0}:=
\{z\in \mathbb C: 0<|z|<\epsilon_0\}$.
If a curve $\gamma: (0,1) \to \Delta^*_{\epsilon_0}$ defined on the unit open interval $(0,1)$
is $\mathcal{C}^1$-smooth such that $\lim_{t \downarrow 0} \gamma(t)= 0$ then, for any
positive integer $n$, the function
\begin{equation}\label{eq44}
\text{Re} \Big[\gamma'(t) \frac{ P'(\gamma(t))}{P^{n+1}(\gamma(t))}\Big]
\end{equation}
cannot be bounded on $(0,1)$.
\end{lemma}

\begin{proof} Suppose that there exists such a $\mathcal{C}^1$-smooth curve
 $\gamma: (0,1)\to \Delta^*_{\epsilon_0}$. Let
\begin{equation*}
F(z):=
\begin{cases}
\dfrac{1}{2}\log P(z) & \text{if  } n=0 \\
\dfrac{-1}{2n P^n(z)} & \text{if  } n\geq 1
\end{cases}
\end{equation*}
and let $u(t):=F(\gamma(t)), \; t\in (0,1)$. (\ref{eq44}) implies that
$$
u'(t)=\text{Re } \Big[\gamma'(t) \frac{ P'(\gamma(t))}{P^{n+1}(\gamma(t))}\Big], \quad \forall t\in (0,1).
$$
Since $u'(t)$ is bounded on $(0,1)$, $u(t)$ also has to be bounded on $(0,1)$.  But this last
is impossible since $u(t)=F(\gamma(t)) \to -\infty$ as $t\downarrow 0$. The lemma is proved.
\end{proof}

This lemma shows in particular that the function ${P'(z)}/{P(z)}$ is unbounded
along any smooth curve  $\gamma: (0,1)\to \Delta^*_{\epsilon_0}\; (\epsilon_0>0)$ such that $\gamma'$ stays bounded on $(0,1)$ and satisfies $\lim_{t\downarrow 0}\gamma(t)=0$.  It has generally been expected that, when a real-valued smooth function $f(t)$ of real variable $t$ near $0$ vanishes to infinite order at $0$, $\lim_{t\downarrow 0}\frac{f'(t)}{f(t)}=\infty$ has to hold and hence the above lemma would have to follow.  However, such a quick expectation is not valid. We present an example here.

\begin{example}
Let $g: (0,1)\to \mathbb R$ be a real valued $\mathcal{C}^\infty$-smooth
function satisfying
\begin{enumerate}
\item[(\romannumeral1)] $g(t)\equiv -2n$ on the closed interval
$ \Big[\dfrac{1}{n+1}\Big(1+\dfrac{1}{3n}\Big),\dfrac{1}{n+1}\Big(1+\dfrac{2}{3n}\Big) \Big]$ for $n=1,2,\cdots$;
\item[(\romannumeral2)] $\dfrac{-2}{t}<g(t)<\dfrac{-1}{t}$ for every $t\in (0,1)$;
\item[(\romannumeral3)] for each $k\in \mathbb N$ there exists $d(k)>0$
such that $|g^{(k)}(t)|\leq \dfrac{1}{t^{d(k)}},\; t\in (0,1)$.
\end{enumerate}
Now let
\begin{equation*}
P(z):=
\begin{cases}
\exp(g(|z|^2)) & \text{if  } 0<|z|<1 \\
0  &\text{if  } z=0,
\end{cases}
\end{equation*}
then this is a $\mathcal{C}^\infty $ function on the open unit disc $ \Delta $ that
vanishes to infinite order at the origin.
However, ${P'(z)}/{P(z)}$ does not tend to $\infty$ as $z \to 0$.
\end{example}

\begin{lemma}\label{l3}
If $a, b$ are complex numbers and if $g_1, g_2$ are smooth functions defined on the punctured
disc $\Delta^*_{\epsilon_0}:=\{0 < |z|< \epsilon_0\}$ with sufficiently small radius satisfying:
\begin{itemize}
\item[(A1)] $g_1(z) = O(|z|^\ell)$ and $g_2(z) = o(|z|^m)$, and
\item[(A2)] $\text{Re} \Big[a z^m+\frac{1}{P^n(z)}\Big(b z^\ell \frac{P'(z)}{P(z)}
+g_1(z) \Big)\Big]= g_2(z)$ for every $z \in \Delta^*_{\epsilon_0}$
\end{itemize}
for any nonnegative integers $\ell, m$ and $n$ except for the following two cases
\begin{itemize}
\item[(E1)] $\ell=1$ and $\text{Re }b = 0$, and
\item[(E2)] $m=0$ and $\text{Re } a = 0$
\end{itemize}
then $ab=0$.
\end{lemma}

\begin{proof} We shall prove the method of contradiction.  Suppose that there exist
non-zero complex numbers $a,b\in \mathbb C^*$ such that the identity in (A2) holds with the smooth functions $g_1$ and $g_2$ satisfying the growth conditions specified in (A1).

Denote by $F(z):=\dfrac{1}{2}\log P(z) $.
\medskip

\noindent
{\bf Case 1.} {\boldmath $\ell=0$:}

Let $u(t):=F(bt),~(0<t< \delta_0)$ with $\delta_0$ sufficiently small.
By (A2), it follows that $u'(t)$ is bounded on the interval $(0, \delta_0)$.
Integration shows that $u(t)$ is also bounded on $(0, \delta_0)$. But this is impossible since
$u(t)\to -\infty$ as $t\downarrow 0$.
\medskip

\noindent
{\bf Case 2.} {\boldmath $\ell=1$:}

 Let $\gamma(t):=e^{bt}, \; t\in (-\infty,+\infty)$.
Then $|\gamma(t)|=e^{ b_1 t}$ and
 $\gamma'(t)=b\gamma(t)$, where $b_1=\text{Re} (b)$. By (E1), we have $b_1\ne 0$. Assume momentarily that $b_1<0$.

Denote by $u(t):= F(\gamma(t))$ for $t\geq t_0$ with $t_0>0$ sufficiently large. It follows
by (A2) that $u'(t)$ is bounded on $(t_0,+\infty)$.
Therefore there exists a constant $A>0$ such that
$ |u(t)|\leq A|b_1| t =A\log \dfrac{1}{|\gamma(t)|}$ for all $t>t_0$.
Hence we obtain, for all $t>t_0$, that
$\log P(\gamma (t))=2u(t)\geq -2A \log \dfrac{1}{|\gamma(t)|} $, and thus
$$
P(\gamma(t))\geq |\gamma(t)|^{2A}, \; t\geq t_0.
$$
Hence we arrive at
$$
\lim_{t\to +\infty} \frac{P(\gamma(t))}{|\gamma(t)|^{2A+1}}=+\infty,
$$
which is impossible since $P$ vanishes to infinite order at $0$.  The case $b_1 >0$ is similar, with considering the side $t<0$ instead.
\medskip

\noindent
{\bf Case 3.} {\boldmath $\ell=k+1\geq 2$:}

Choose $c\in \mathbb C$ such that
$c-kbt\in \mathbb C\setminus [0,+\infty) $ for all $t\in (-\infty, +\infty)$.
Let $\gamma(t):= \sqrt[-k]{c-kbt}=\sqrt[-k]{|c-kbt|} e^{-i arg (c-kbt)/k}$, $0<arg (c-kbt)<2\pi$.
Note that $|\gamma(t)|\approx \dfrac{1}{|t|^{1/k}}$ for $|t|\geq t_0$ with $t_0>0$ big enough.
Let $u(t):= F(\gamma(t))$. It follows from (A2) that
\begin{equation}\label{eq77}
u'(t)=-P^n(\gamma(t))(\text{Re} (a \gamma^m(t) +o(|\gamma(t)|^m))+O(|\gamma(t)|^{k+1}).
\end{equation}
We now consider the following.
\medskip

\underbar{\it Subcase 3.1}: $n\geq 1.$

Since $P$ vanishes to infinite order at the origin,
 (\ref{eq77}) and the discussion above imply
\begin{equation*}
\begin{split}
|u'(t)|&\lesssim P^n(\gamma(t)) |\gamma(t)|^m+\frac{1}{t^{1+/k}}\\
&\lesssim P^n(\gamma(t))+\frac{1}{t^{1+1/k}}\\
&\lesssim \frac{P^n(\gamma(t))}{|\gamma(t)|^{2k}} \frac{1}{t^2}+\frac{1}{t^{1+1/k}}\\
&\lesssim  \frac{1}{t^2}+\frac{1}{t^{1+1/k}} \\
&\lesssim  \frac{1}{t^{1+1/k}},
\end{split}
\end{equation*}
for all $t\geq t_0$. This in turn yields
\begin{equation*}
\begin{split}
|u(t)|&\lesssim |u(t_0)|+ \int_{t_0}^t   \frac{1}{s^{1+1/k}}ds\\
 &\lesssim |u(t_0)|
+k \Big(\frac{1}{t_0^{1/k}}-\frac{1}{t^{1/k}}\Big)\\
&\lesssim 1
\end{split}
\end{equation*}
for all $t>t_0$. This is a contradiction, because $\lim_{t\to\infty} u(t)=-\infty$.
\medskip

\noindent
\underbar{\it Subcase 3.2}: $n=0$.

We again divide the argument in 4 sub-subcases.
\medskip

\underbar{\it Subcase 3.2.1}: $m/k>1$.

It follows from (\ref{eq77}) that
\begin{equation*}
|u'(t)|\lesssim  \frac{1}{t^{m/k}}+\frac{1}{t^{1+1/k}}\\
\end{equation*}
for all $t\geq t_0$. Hence, we get

\begin{equation*}
\begin{split}
|u(t)|&\lesssim |u(t_0)|+ \int_{t_0}^t   \Big(\frac{1}{s^{m/k}}+\frac{1}{s^{1+1/k}}\Big)ds\\
 &\lesssim |u(t_0)|+ \frac{k}{m-k}  \Big(\frac{1}{t_0^{m/k-1}}-\frac{1}{t^{m/k-1}}\Big)
+k \Big(\frac{1}{t_0^{1/k}}-\frac{1}{t^{1/k}}\Big)\\
&\lesssim 1
\end{split}
\end{equation*}
for all $t>t_0$, a contradiction.
\medskip

\underbar{\it Subcase 3.2.2}: ${m}/{k}=1$.

Here, (\ref{eq77}) again implies
\begin{equation*}
|u'(t)| \lesssim  \frac{1}{t}+\frac{1}{t^{1+1/k}} \lesssim  \frac{1}{t},
\end{equation*}
for all $t\geq t_0$. Consequently,

\begin{equation*}
\begin{split}
|u(t)|&\lesssim |u(t_0)|+ \int_{t_0}^t   \frac{1}{s}\ ds\\
 &\lesssim |u(t_0)|+  (\log t-\log t_0)\\
 &\lesssim \log t \lesssim \log \frac{1}{|\gamma(t)|},
\end{split}
\end{equation*}
for all $t>t_0$. Therefore there exists a constant $A>0$ such that
$ |u(t)|\leq A\log \dfrac{1}{|\gamma(t)|}$ for all $t>t_0$.
Hence for all $t>t_0$,
$\log P(\gamma (t))=2u(t)\geq -2A \log \dfrac{1}{|\gamma(t)|} $, and thus
$$
P(\gamma(t))\geq |\gamma(t)|^{2A}, \; \forall t\geq t_0.
$$
This implies
$$ \lim_{t\to +\infty} \frac{P(\gamma(t))}{|\gamma(t)|^{2A+1}}=+\infty,$$
impossible since $P$ vanishes to infinite order at $0$.
\medskip

\underbar{\it Subcase 3.2.3}: $m=0$.

Let $h(t):=u(t)+\text{Re}(a) t$. Recall that in this case we have (E2) which says $\text{Re }a \neq 0$. Assume momentarily that $\text{Re}(a)<0$. (The case that $\text{Re}(a)>0$ will follow by a similar argument.)

By (\ref{eq77}), there is a constant $B>0$ that
$$
|h'(t)|\leq \frac{1}{2}|\text{Re}(a)| +B\frac{1}{t^{1+1/k}}.
$$
Therefore,
\begin{equation*}
\begin{split}
|h(t)|&\leq |h(t_0)|+ \frac{1}{2}|\text{Re}(a)| (t-t_0)+ B\int_{t_0}^t   \frac{1}{s^{1+1/k}}ds\\
 &\leq |h(t_0)|+ \frac{1}{2}|\text{Re}(a)| (t-t_0)+k B(\frac{1}{t_0^{1/k}}-\frac{1}{t^{1/k}})\\
\end{split}
\end{equation*}
for all $t>t_0$. Thus
\begin{equation*}
\begin{split}
u(t) &\geq -\text{Re}(a) t-|h(t)|\\
    &\geq |\text{Re}(a)| t-|h(t_0)|- \frac{1}{2}|\text{Re}(a)| (t-t_0)-
k B(\frac{1}{t_0^{1/k}}-\frac{1}{t^{1/k}})\\
     &\gtrsim t
\end{split}
\end{equation*}
for all $t>t_0$. It means that $u(t)\to +\infty$ as $t\to +\infty$, absurd.
\medskip

\underbar{\it Subcase 3.2.4}: $0<\frac{m}{k}<1$.

Notice first that $k\geq 2$.   Let $\tau=e^{i 2\pi/k}$ and $\gamma_j(t) :=\tau^{-j} \gamma (t)$
for  $j=0, 1\cdots,k-1$.
Then $\gamma_j'(t)= b \gamma_j^{k+1}(t)$ and $\gamma_j(t)\to 0$
as $t\to \infty$.

Set $u_j(t):= F(\gamma_j(t))$. Assume for a moment that $m$ and $k$ are relatively prime. (In the end, it will become obvious that this assumption can be taken without loss of generality.) Then $\tau^{m}$ is a primitive $k$-th root of unity. Therefore there
exist $j_0,j_1\in \{1,\cdots,k-1\} $ such that $\pi/2<arg(\tau^{mj_0})\leq\pi$ and\break
$-\pi\leq arg(\tau^{mj_1})<-\pi/2$. Hence, it follows that there exists
$j\in \{0,\cdots,k-1\}$ such that $\cos\big(arg(a/b)+\frac{k-m}{k} arg(-b)-2\pi mj/k\big)>0$.
Denote by
$$
A:=\frac{|a|}{(k-m)|b|}\cos\Big(arg(a/b)+ \frac{k-m}{k}arg(-b)-2\pi m j/k\Big)>0,
$$
a positive constant. Now let
$$
h_j(t):= u_j(t)+\text{Re}(\tau^{-mj} \frac{a}{-b(k-m)}(c-kbt)^{1-m/k}).
$$
Note that $arg(c-kbt)\to arg (-b) $ as $t\to +\infty$.
Hence it follows from (\ref{eq77}) that there exist positive constants $B$ and $t_0$ such that
$$
|h_j'(t)|\leq \frac{k-m}{4k} A (k|b|)^{1-m/k}\frac{1}{t^{m/k}}+\frac{B}{t^{1+1/k}}
$$
and
\begin{multline*}
\cos\Big(\arg(a/b)+\frac{k-m}{k}\arg(c-kbt)-2mj\pi /k\Big) \\
\geq \frac{1}{2}\cos\Big(\arg(a/b)+ \frac{k-m}{k}\arg(-b)-2mj \pi /k\Big)
\end{multline*}
for every $t\geq t_0$. Thus we have
 \begin{equation*}
\begin{split}
|h_j(t)|&\leq |h_j(t_0)|+A(k|b|)^{1-m/k}\frac{k-m}{4k} \int_{t_0}^t s^{-m/k}ds
+B\int_{t_0}^ts^{-1-1/k}ds\\
        & \leq |h_j(t_0)|+\frac{A}{4}(k|b|)^{1-m/k} (t^{1-m/k}-t_0^{1-m/k})
+kB(t_0^{-1/k}-t^{-1/k})
\end{split}
\end{equation*}
for $t>t_0$. Hence
\begin{equation*}
\begin{split}
u_j(t)&\geq -\text{Re} (\frac{a \tau^{-mj}}{-kb(1-m/k)} (c-kbt)^{1-m/k}) -|h_j(t)|\\
& \geq \frac{|a|}{|b|(k-m))} |c-kbt|^{1-m/k}\cos\Big( arg(a/b)\\
&\qquad +\frac{(k-m)arg(c-kbt)- 2mj\pi }{k}\Big)- |h_j(t_0)|\\
&\qquad -\frac{A}{4} (k|b|)^{1-m/k}(t^{1-m/k}-t_0^{1-m/k})-kB(t_0^{-1/k}-t^{-1/k})\\
&\geq \frac{A}{2}|c-kbt|^{1-m/k}- |h_j(t_0)|\\
&\qquad -\frac{A}{4}(k|b|)^{1-m/k} (t^{1-m/k}-t_0^{1-m/k})-kB(t_0^{-1/k}-t^{-1/k})\\
&\gtrsim t
\end{split}
\end{equation*}
for $t>t_0$.  This implies that $u_j(t)\to +\infty$ as $t\to +\infty$,
which is absurd since $\log P(z)\to -\infty$ as $z \to 0$.
\medskip

Hence all the cases are covered, and the proof of Lemma \ref{l3} is finally complete.
\end{proof}

\begin{lemma} \label{l6}
Suppose that $R$ is a real-valued $\mathcal{C}^1$-smooth function defined on the disc
$\Delta_{\epsilon} := \{z\in \mathbb C: |z|<\epsilon\}$ for some $\epsilon>0$.
Then,  $\text{Re }(i z (\partial R/\partial z)(z))=0$ for all
$z\in\Delta_\epsilon$ if and only if $R(z)=R(|z|)$.
\end{lemma}
\begin{proof}

\noindent
Let $r$ be an artribrary number such that $0<r<\epsilon$ and let
$v(t):= R(r e^{it} )$. Since $\text{Re }(i z (\partial R/\partial z)(z))=0 $, $v'(t)=0$ for every $t\in \mathbb R$.
Thus $v(t)\equiv v(0)$ and hence $R(z)=R(|z|)$.
This completes the proof as the converse is obvious.
\end{proof}

\begin{lemma} \label{l8}
If $R$ is a real-valued $\mathcal{C}^1$-smooth function defined on an open neighborhood,
say $U$, of the origin in $\mathbb C$, then on every circle $\{z\in\mathbb C\colon |z|=r\}$ contained in $U$
the function $\text{Re}(i z R'(z))$ either identically zero, or must change sign.
\end{lemma}

\begin{proof}
Since $\frac{d}{dt} R(r e^{it}) = \text{Re }[i re^{it} (\partial R/\partial z)(re^{it})]$, the function
$R(r e^{it})$ cannot stay periodic in the real-variable $t$, unless
$\text{Re }[i re^{it} (\partial R/\partial z)(re^{it})]$ changes its sign.
\end{proof}

\begin{lemma}\label{l7} If $b$ is a complex number satisfying
\begin{equation}\label{eqqtn1}
\text{Re} (b z^k P'(z))=0,
\end{equation}
for some nonnegative integer $k$, except the case $k= 1$ and $\text{Re} (b)=0$,
on $z \in \Delta_\epsilon$ with $\epsilon>0$, then $b=0$.
\end{lemma}
\begin{proof}
We consider three following cases.
\smallskip

\underbar{Case (\romannumeral1)}: $k=0$. Let $u(t):=P(bt)$, $t\in (-\delta,+\delta)$ for some $\delta>0$. It follows
from (\ref{eqqtn1}) that $u'(t)\equiv 0$ on $(-\delta,+\delta)$, thus $u(t)\equiv u(0)=0$ on
$(-\delta,+\delta)$. Impossible.
\smallskip

\underbar{Case (\romannumeral2)}: $k=1$. Assume momentarily that $b_1=\text{Re}b <0$.
For each $ c \in \mathbb C^*$ let $u(t):=P(c e^{bt})$ for all $t\geq t_0$
with $t_0>0$ sufficiently large. It follows by (\ref{eqqtn1}) that $u'(t)\equiv 0$ on $(t_0,+\infty)$. Hence $u(t)\equiv 0 $ and consequently $P\equiv 0$ on $|z|<\epsilon_0$, absurd.
\smallskip

\underbar{Case (\romannumeral3)}:  $k=\ell+1\geq 1$. Choose $c\in \mathbb C$ such that
$c-\ell bt\in \mathbb C\setminus [0,+\infty) $
for  every $t\in (-\infty,+\infty)$.
 Let $\gamma(t):=\sqrt[-\ell]{c-\ell bt}=\sqrt[-\ell]{|c-\ell bt|} e^{-i arg (c-\ell bt)/\ell}$, $0<arg (c-\ell bt)<2\pi$.
Let $u(t):= P(\gamma(t))$. It follows from (\ref{eqqtn1}) that $u'(t)\equiv 0$ on $(t_0,+\infty)$,
for some $t_0>0$ sufficiently large,  and therefore $u(t)$ is constant on $(t_0,+\infty)$.
Since $\lim_{t\to +\infty} u(t)=P(0)=0$, $P(\gamma(t))\equiv 0$ for all $t>t_0$,
which is again impossible.
\end{proof}

\subsection{Holomorphic tangent vector fields: Proof of Theorem \ref{T2}}
The CR hypersurface germ $(M,0)$ at the origin in $\CC^2$ under consideration is defined by the equation $\rho(z_1, z_2) = 0$ where
$$
\rho (z_1, z_2) = \text{Re }z_1 + P(z_2) + (\text{Im }z_1)\ Q(z_2, \text{Im }z_1) = 0,
$$
where $P, Q$ are $\mathcal{C}^\infty$ smooth functions satisfying the three conditions specified in the hypothesis of Theorem 3, stated in Section 2.  Recall that $P$ vanishes to infinite order at $z_2=0$ in particular.

Then consider a holomorphic vector field $H=h_1(z_1,z_2)\frac{\partial}{\partial z_1}+h_2(z_1,z_2)\frac{\partial}{\partial z_2}$ defined on a neighborhood of the origin. We only consider $H$ that is tangent to $M$, which means that they satisfy the identity
\begin{equation}\label{eq221}
(\text{Re } H) \rho(z)=0,\; \forall z \in M.
\end{equation}

The goal is to characterize all such $H$.
\bigskip

Since
\begin{equation*}
\begin{split}
 \rho_{z_1}(z_1,z_2)&= \frac{1}{2} +\frac{1}{2i}Q(z_2, \text{Im } z_1)+
\text{Im } z_1Q_{z_1}(z_2, \text{Im} z_1),\\
\rho_{z_2}(z_1,z_2)&= P'(z_2)+\text{Im } z_1 Q_{z_2}(z_2, \text{Im} z_1).
\end{split}
\end{equation*}
the equation (\ref{eq221}) is re-written as
\begin{equation}\label{eq23}
\begin{split}
\text{Re} \Big [&\big( \frac{1}{2} +\frac{1}{2i}Q(z_2, \text{Im } z_1)+
\text{Im } z_1Q_{z_1}(z_2, \text{Im } z_1)\big)h_1(z_1,z_2)+\\
&\qquad +(P'(z_2)+\text{Im } z_1Q_{z_2}(z_2, \text{Im } z_1)) h_2(z_1,z_2) \Big ]=0,
\end{split}
\end{equation}
for all $(z_1,z_2)\in M$.

Since $(it-P(z_2)-tQ(z_2,t), z_2)\in M$ for any $t \in \RR$ with $|t|<\delta$, the equation again
takes the new form:
\begin{equation}\label{eq24}
\begin{split}
&\text{Re} \Big [\big(\frac{1}{2}+\frac{1}{2i}Q(z_2,t)+
tQ_{z_1}(z_2,t)\big) h_1(it-P(z_2)-tQ(z_2,t)), z_2)\\
&+\big(P'(z_2)+tQ_{z_2}(z_2,t)\big)h_2(it-P(z_2)-tQ(z_2,t),z_2) \Big ]=0.
\end{split}
\end{equation}
Expand $h_1$ and $h_2$ into the Taylor series at the origin so that
$$
h_1(z_1,z_2)=\sum\limits_{j,k=0}^\infty a_{jk} z_1^j z_2^k
\text{ and }
h_2(z_1,z_2)=\sum\limits_{j,k=0}^\infty b_{jk} z_1^jz_2^k.
$$
Note that $a_{00}=b_{00}=0$ since $h_1(0,0)=h_2(0,0)=0$.

Notice that we may choose $t=\alpha P(z_2)$ in (\ref{eq24})
(with $\alpha \in \mathbb R$ to be chosen later). Then one gets
\begin{equation}\label{eq27}
\begin{split}
\text{Re} \Big [\big(\frac{1}{2}&+\frac{1}{2i}Q(z_2,\alpha P(z_2))+
\alpha P(z_2)Q_{z_1}(z_2,\alpha P(z_2))\big)\times \\
& \quad h_1\big(i\alpha P(z_2)- P(z_2)- \alpha P(z_2)Q(z_2,\alpha P(z_2)), z_2\big)\\
&+\big(P'(z_2)+ \alpha  P(z_2) Q_{z_2}(z_2,\alpha  P(z_2))\big)\times\\
& \quad h_2(i \alpha  P(z_2)- P(z_2)-\alpha  P(z_2)Q(z_2,\alpha P(z_2)),z_2)
 \Big ]=0.
\end{split}
\end{equation}
for all $z_2$ with $|z_2|<\epsilon_0$, for some positive $\epsilon_0$ sufficiently small.
\medskip

We now prove that $h_1\equiv 0$ on a neighborhood of $(0,0)$ in $\mathbb C^2$.
\medskip

Assume the contrary that $h_1\not \equiv 0$. Then there exist non-negative integers
$j,k$ such that $a_{jk}\ne 0$ and the largest term in
\begin{equation*}
\begin{split}
&\text{Re}\Big [\Big (\frac{1}{2}+\frac{1}{2i}Q(z_2,\alpha  P(z_2))
+\alpha P(z_2)Q_{z_1}(z_2,\alpha P(z_2))\Big) \times \\
& h_1(i\alpha P(z_2)-P(z_2)-\alpha P(z_2) Q(z_2,\alpha  P(z_2)), z_2)\Big ]
\end{split}
\end{equation*}
is $\text{Re}\Big [\dfrac{1}{2}a_{jk}(i\alpha -1)^jz_2^k(P(z_2))^j\Big]$, where the ``largest'' is measured in terms of the speed of growth.
We note that in the case $k=0$ and $\text{Re} a_{j0}=0$, $\alpha $ can be chosen in such a way that
$\text{Re}(a_{j0}(i\alpha-1)^j) \ne 0 $. Therefore there are nonnegative integers $m,n$ such that $b_{mn}\ne 0$ and that the biggest term in
\begin{equation*}
\begin{split}
\text{Re}\Big [&\big(P'(z_2)+ \alpha P(z_2)Q_{z_2}(z_2,\alpha P(z_2))\big)\\
&h_2(i\alpha P(z_2)-P(z_2)-\alpha P(z_2)Q(z_2,\alpha P(z_2)),z_2)\Big ]
\end{split}
\end{equation*}
is $\text{Re}\Big[ b_{mn}(i\alpha -1)^mz^n_2(P'(z_2)
+\alpha P(z_2) Q_{z_2}(z_2, \alpha P(z_2)))(P(z_2))^m \Big ]$ for some $m,n$
with $b_{mn}\ne 0$. By (\ref{eq27}) we get
\begin{equation}\label{eq28}
\begin{split}
\text{Re } &\Big[\frac{1}{2} a_{jk}(i\alpha -1)^j(P(z_2))^jz_2^k+
 b_{mn}(i\alpha -1)^mz^n_2\times \\
& (P'(z_2)+\alpha P(z_2)Q_{z_2}(z_2, \alpha P(z_2))) (P(z_2))^m  \Big ]=o(P(z_2)^j|z_2|^k),
\end{split}
 \end{equation}
for all $|z_2|<\epsilon_0$. Observe that $j> m$. Note also that, if $k=0$
and $\text{Re}(a_{j0})\ne 0$, then letting $\alpha =0$ in (\ref{eq28})
we get $\text{Re}(a_{j0}+b_{0m}z_2^n P'(z_2)/P^{j-m}(z_2)) \to 0$ as $z_2\to 0$,
which is not possible because of Lemmas \ref{l1} and \ref{l22}.
Hence, we may assume that $\text{Re}a_{j0}=0$ for the case $k=0$.

We now divide the argument into two cases as follows:
\smallskip

{\bf Case 1.} {\boldmath $m=0$.}
In addition to this condition, if $n>1$, or if $n=1$ and $\text{Re}(b_{01})\ne 0$,
then (\ref{eq28}) contradicts Lemma \ref{l22}. Therefore, we may assume that $n=1$ and $\text{Re} b_{01}=0$. Choose $\alpha_1,\alpha_2\in \mathbb R$ with
$\alpha_1\ne \alpha_2$ such that $(\ref{eq28})$ holds for $\alpha=\alpha_\ell\; (\ell=1,2)$; thus one obtains two equations.
Subtracting one from the other yields:
\begin{equation*}
\begin{split}
f(z_2):=&\text{Re}\Big [ a_{jk}((i\alpha_1 -1)^j-(i\alpha_2 -1)^j)z_2^k+\\
&b_{01}z_2 \dfrac{\alpha_1Q_{z_2}(z_2, \alpha_1 P(z_2))-\alpha_2
Q_{z_2}(z_2, \alpha_2 P(z_2))}{P^{j-1}(z_2)} \Big ]
=o(|z_2|^k)
\end{split}
\end{equation*}
for every $z$ satisfying $0<|z|<\epsilon_0$.

If $j=1$ then, taking $\lim\limits_{\delta \to  0^+} \dfrac{1}{\delta^k}f(\delta z_2)$,  we obtain
$$
\text{Re }\Big [ i a_{1k}z_2^k+ b_{01} z_2 \psi(z_2)\Big ] =0,
$$
where $\psi$ is a homogeneous polynomial of degree $k-1$. Note that this identity implies that $k\not= 0$. So $k-1>0$.  Now, the same identity says that the homogenous polynomial $\psi$ must contain $c z^{k-1}$.  But this is impossible, since $\psi(z_2)$ comes from $Q_{z_2}(z_2,0)$ which has no harmonic terms.
\medskip

Now we consider the case $j>1$. Taking $\lim\limits_{\delta \to  0^+} \dfrac{1}{\delta^k}f(\delta z_2)$  we obtain
\begin{equation}\label{eq30}
 \text{Re}\Big [ a_{jk}((i\alpha_1 -1)^j-(i\alpha_2 -1)^j)z_2^k+
b (\alpha_1^\ell-\alpha_2^\ell) z_2^k\Big ] =0,
\end{equation}
where $b\in \mathbb C^*$ and $\ell \geq 1$ are both independent of $\alpha_1$ and $\alpha_2$.
Note that $\ell\geq 2$ for the case $k=0$. Indeed, suppose otherwise that $k=0$ and $\ell=1$.
Then $\lim_{z_2\to 0}\text{Re}\big(b_{01}z_2 \frac{Q_{z_2}(z_2,0)}{P^{j-1}(z_2)}\big) =a$,
where $0\neq a\in \mathbb R$. This contradicts Lemma \ref{l8}.
\smallskip

It follows by (\ref{eq30}) that
$$
a_{jk}((i\alpha_1 -1)^j-(i\alpha_2 -1)^j)+
b (\alpha_1^\ell-\alpha_2^\ell) =0$$
for $k\geq 1$,
and
$$ \text{Re}\Big[a_{j0}((i\alpha_1 -1)^j-(i\alpha_2 -1)^j)+
b (\alpha_1^\ell-\alpha_2^\ell)\Big] =0$$
for $k=0$. Since $\alpha_1$ can be arbitrarily chosen in $\mathbb R$ and note that
$\text{Re } (a_{j0})=0$, taking the $N$-th derivative of  both sides of above equations
with respect to $\alpha_1$ at $\alpha_1=0$, where $N=1$ if $\ell\geq 2$ and $N=2$ if $\ell=1$, we obtain that $a_{jk}=0$, which is absurd.
\medskip

\noindent
{\bf Case 2.} {\boldmath $m\geq 1$.}
If $n=1$, then the number $\alpha$ can also be chosen such that
$\text{Re }(b_{m1}(i\alpha-1)^m)\ne 0$. Therefore, $(\ref{eq28})$ contradicts
Lemma \ref{l3}, and thus $h_1\equiv 0$ on a neighborhood of $(0,0)$ in $\mathbb C^2$.
\medskip

Since $h_1\equiv 0$, it follows from (\ref{eq24}) with $t=0$ that
$$
\text{Re }\Big [ \sum_{m,n=0}^\infty  b_{mn} z^n_2P'(z_2)\Big ]=0,
$$
for every $z_2$ satisfying $|z_2|<\epsilon_0$, for some $\epsilon_0>0$ sufficiently small. By Lemmas \ref{l6} and \ref{l7}, we conclude that $b_{mn}=0$ for every
$m,n\ge 0$ except the case that $m=0$ and $n=1$. In this last case $b_{01}=i \beta$ for some nonzero real number $\beta$ and
$P$ is rotationally-symmetric. Moreover, (\ref{eq24}) yields that
$\text{Re }(i z_2Q_{z_2}(z_2,t))=0$ for every $z$ with $|z|<\epsilon_0$ and $t$ with $-\delta_0<t<\delta_0$, for sufficiently small positive real constants
$\epsilon_0$ and $\delta_0$.
This of course implies that $Q(z_2,t)$ is radially symmetric in $z_2$
by Lemma \ref{l6}.

Altogether, the proof of Theorem \ref{T2} is complete. \hfill $\Box\;$


\begin{thebibliography}{99}
\bibitem{B-P1}
E. Bedford and S. Pinchuk:  Domains in $\mathbb C^2$ with noncompact groups of
automorphisms,  {\it Math. USSR Sbornik} 63 (1989), 141--151.

\bibitem{B-P2}
E. Bedford and S. Pinchuk: Domains in $\mathbb C^{n+1}$ with noncompact automorphism
group, {\it J. Geom. Anal.} 1 (1991), 165--191.

\bibitem{B-P3}
E. Bedford and S. Pinchuk: Domains in $\mathbb C^2$
with noncompact automorphism groups, {\it Indiana Univ. Math. J.} 47 (1998), 199-222.

\bibitem{Bell} S. Bell: Compactness of families of holomorphic mappings up to the boundary,
{\it Lect. Notes in Math.} Vol. 1268, Springer-Verlag, 1987, 29--42.

\bibitem{Bell-Lig} S. Bell and E. Ligocka: A simplification and extension of
Fefferman's theorem on biholomorphic mappings, {\it Invent. Math.} 57 (1980), no. 3, 283--289.

\bibitem{Ber} F. Berteloot, Characterization of models in $\mathbb C^2$
by their automorphism groups, {\it Internat. J. Math.} 5 (1994), 619--634.

\bibitem{By} J. Byun and H. Gaussier: On the compactness of the automorphism group
of a domain, {\it C. R. Acad. Sci. Paris}, Ser. 1341 (2005), 545--548.

\bibitem{By1} J. Byun, J.-C. Joo and M. Song: The characterization
of holomorphic vector fields vanishing at an infinite type point,
{\it J. Math. Anal. Appl.} 387 (2012), 667--675.

\bibitem{D} J. P. D'Angelo: Real hypersurfaces, orders of contact, and applications,
{\it Ann. Math.} 115 (1982), 615--637.

\bibitem{Fef} C. Fefferman: The Bergman kernel and biholomorphic mappings
of pseudoconvex domains, {\it Invent. Math.} 26 (1974), 1--65.

\bibitem{GK} R. Greene and S. G. Krantz: Techniques for studying
automorphisms of weakly pseudoconvex domains, {\it Math. Notes},
Vol 38, Princeton Univ. Press, Princeton, NJ, 1993, 389--410.

\bibitem{IK} A. Isaev and S. G. Krantz: Domains with non-compact automorphism group:
A survey, {\it Adv. Math.} 146 (1999), 1--38.

\bibitem{Ka} H. Kang: Holomorphic automorphisms of certain class of domains of infinite type,
{\it Tohoku Math. J.} 46 (1994), 345--422.

\bibitem{Ki1} K.-T. Kim: On a boundary point repelling automorphism orbits,
{\it J. Math. Anal. Appl.} 179 (1993), 463--482.

\bibitem{Ki2} K.-T. Kim and S. G. Krantz: Convex scaling and domains with non-compact
automorphism group, {\it Illinois J. Math.} 45 (2001), 1273--1299.

\bibitem{Ki3} K.-T. Kim and S. G. Krantz: Some new results on domains in
complex space with non-compact automorphism group, {\it J. Math. Anal. Appl.}
281 (2003), 417--424.

\bibitem{Kim} K.-T. Kim: Domains in $\mathbb C^n$ with a piecewise
Levi flat boundary which possess a noncompact automorphism group,
{\it Math. Ann.} 292 (1992), no. 4, 575--586.

\bibitem{K-P} K.-T. Kim and A. Pagano: Normal analytic polyhedra
in $\mathbb C^2$ with a noncompact automorphism group, {\it J. Geom. Anal.}
11 (2001), no. 2, 283--293.

\bibitem{K-K-S} K.-T. Kim, S. G. Krantz and A. F. Spiro: Analytic
polyhedra in $\mathbb C^2$ with a non-compact automorphism group,
{\it J. Reine Angew. Math.} 579 (2005), 1--12.

\bibitem{KiYo} K.-T. Kim and J.-C. Yoccoz: CR manifolds admitting a
CR contraction, {\it J. Geom. Anal.} 21 (2011), no. 2, 476--493.

\bibitem{Kol} M. Kolar: Normal forms for hypersurfaces of infinite type
in $\mathbb C^2$, {\it Math. Res. Lett.} 12 (2005), p. 897--910.

\bibitem{La} M. Landucci: The automorphism group of domains with boundary
 points of infinite type, {\it Illinois J. Math.} 48 (2004), 33--40.

\bibitem{Ninh} Ninh Van Thu and Chu Van Tiep: On the nonexistence of
parabolic boundary points of certain domains in $\mathbb C^2$,
{\it J. Math. Anal. Appl.} 389 (2012), 908--914.

\bibitem{R} J.-P. Rosay: Sur une caracterisation de la boule parmi les domaines
de $\mathbb C^n$ par son groupe d'automorphismes, {\it Ann. Inst. Fourier} 29 (4) (1979), 91--97.

\bibitem{W} B. Wong: Characterization of the ball in $\mathbb C^n$ by its automorphism group,
{\it Invent. Math.} 41 (1977), 253--257.
\end{thebibliography}
\end{document}